\documentclass{amsart}

\usepackage{lineno,hyperref}
\usepackage{amssymb}
\usepackage{amsmath}
\usepackage{amsthm}
\usepackage{mathrsfs}
\usepackage{bbm}
\usepackage{cite}
\usepackage[all]{xy}
\usepackage{graphicx}

\newtheorem{thm}{Theorem}[section]
\newtheorem{cor}[thm]{Corollary}
\newtheorem{defn}[thm]{Definition}
\newtheorem{lem}[thm]{Lemma}
\newtheorem{prop}[thm]{Proposition}
\newtheorem{exa}[thm]{Example}
\newtheorem{rem}[thm]{Remark}

\title{Growth theorems in slice analysis of several variables}
\author{Guangbin Ren, Ting  Yang}
\email[G.~Ren]{rengb@ustc.edu.cn}
\email[T.~Yang]{@mail.ustc.edu.cn}
\date{\today}
\thanks{The first author is supported by the NNSF  of China (11771412).}

\begin{document}
\maketitle
\markboth{Growth theorems in slice analysis of  several variables}
{Growth theorems in slice analysis of  several variables}
\renewcommand{\sectionmark}[1]{}

\begin{abstract}
In this paper, we define a class of slice mappings of several Clifford variables, and the corresponding slice regular mappings. Furthermore,
we establish the growth theorem for slice regular starlike or convex mappings on the unit ball  of several slice  Clifford variables, as well as on the bounded slice domain which is slice starlike and slice circular.
\end{abstract}

\section{Introduction}

Ghiloni and Perotti \cite{Ghiloni2011002}  initiated the study of slice analysis   on the quadratic cones of real alternative algebras;  see \cite{Ghiloni2017003},\cite{Wang2016002} for its recent development.
The theory has some distinguished models on quaternions, real Clifford algebras, and Octonions; see
\cite{Gentili2007001},\cite{Gentili2008002},\cite{Ghiloni2011002},\cite{Ghiloni2017002}, \cite{Colombo2009002},\cite{Colombo2009003},\cite{Colombo2010003},\cite{Gentili2010002}  for the pioneering works.
 It has important applications in the functional calculus for non-commutative operators \cite{Colombo2011006}.

Recently, the slice   theory of  several variables has also been studied \cite{Adams1999001},\cite{Colombo2012003},\cite{Ghiloni2012002},\cite{Wang2014001}.
In this article, we shall provide a new generalization, which enable us to extend slice analysis to higher dimensions as well as to the setting of several slice Clifford variables.
%We shall focus our attention to the growth theorems for slice regular starlike or convex mappings on the unit ball of several slice Clifford variables.

Based on a new convex combination identity in \cite{Wang2017001}, the sharp growth theorems for slice monogenic extensions of univalent functions on the unit disc $D\subset \mathbb{C}$ in the setting of Clifford algebras was established as follows:

Let $f$ be a slice monogenic function on the unit ball $\mathbb{B}$ in the regular quadratic cone $\mathbb{R}^{(m+1)}$  (in this paper, $\mathbb{R}^{(m+1)}$ denotes $\mathbb{R}^{m+1}$  in \cite{Wang2017001}) of $\mathbb{R}_m$ such that its restriction $f_I$ to $\mathbb{B}_I$ is injective and such that $f(\mathbb{B}_I)\subseteq \mathbb{C}_I$ for some $I\in \mathbb{S}_m.$ If $f(0)=0, \ f'(0)=1,$  then
$$\frac{|x|}{(1+|x|)^2}\leq |f(x)|\leq \frac{|x|}{(1-|x|)^2},\qquad x\in \mathbb{B}.$$
Moreover, equality holds for one of these two inequalities at some point $x_0\in \mathbb{B}\setminus\{0\}$ if and only if $f$ is of the form
$f(x)=x(1-xe^{I\theta})^{-*2},\ \forall x\in \mathbb{B},$ for some $\theta \in \mathbb{R}.$ (See \cite{Wang2017001} for details.)

From the classical geometric function theory in higher dimensions,  it is known that the growth theorems fail for the full class of normalized univalent mappings \cite{Graham2003001}. It is Cartan who suggested to consider the subclass of starlike or convex mapping instead.

The aim of this paper is to generalize the growth results for subclasses of normalized univalent mappings on the unit ball in $\mathbb{C}^n$ to the subset of the several slice Clifford variables $(\mathcal{Q}^{m})_s^n.$

Since the convex combination identity mentioned above was only proved for one-dimensional case, we overcome this problem by using a new approach. As an application, we obtain the growth theorems for slice regular starlike or convex mappings on the unit ball of several slice  Clifford variables, as well as on a more general slice domain which is slice starlike and slice circular.

%\newpage{}

\section{Priliminaries}
The real Clifford algebra $\mathbb{R}_m$ is an universal associative algebra over $\mathbb{R}$ generated by $m$ basis elements $e_1,\cdots,e_m,$ subject to the relations
$$e_\imath e_\jmath+e_\jmath e_\imath=-2\delta_{\imath \jmath},\qquad \imath,\jmath=1,\cdots,m.$$
As a real vector space, $\mathbb{R}_m$ has dimension $2^m.$ Each element $x$ in $\mathbb{R}_m$ can be expressed as
$$x=\sum_{A\in \mathcal{P}(m)}x_Ae_A,$$
where
\begin{equation*}
\begin{split}
\mathcal{P}(m)&��=\{(h_1,\cdots, h_r)\in \mathbb{N}^r \ | \ r=1,\cdots,m,\ 1\leq h_1< \cdots < h_r\leq m\}.
\end{split}
\end{equation*}

For each $A=\{(h_1,\cdots,h_r)\in \mathcal{P}(m)\}$,
the coefficients $x_A\in \mathbb{R}$, and the products $e_A:=e_{h_1}e_{h_2}\cdots e_{h_r}$ are the basis elements of the Clifford algebra $\mathbb{R}_m.$  The unit of the Clifford algebra corresponds to $A=\varnothing$, and we set $e_\phi=1$.  As usual, we identify the real numbers field $\mathbb{R}$ with the subalgebra of $\mathbb{R}_m$ generated by the unit.

%The Clifford conjugate of each generator $e_i,\ i=1,2,\cdots,m,$ is defined to be $\bar e_i=-e_i,$ and thus extends to each $e_A$ by setting
%$$\bar e_A=\bar e_{h_r} \bar e_{h_r-1}\cdots\bar e_{h_1}=(-1)^r{e_{h_r}}{e_{h_r-1}}\cdots{e_{h_1}}=(-1)^{\frac{r(r+1)}{2}}e_A.$$
Let $x^c$ be the Clifford conjugate of $x\in \mathbb{R}_m.$ We define
 $$t(x)=x+x^c, \qquad n(x)=x x^c $$
to  be the trace and  the (squared) norm of a Clifford element $x$, respectively.

The quadratic cone of the Clifford algebra $\mathbb{R}_m$ \cite{Ghiloni2012002} is defined by
$$\mathcal{Q}_m:=\mathbb{R}\cup \{x\in \mathbb{R}_m \ | \ t(x)\in \mathbb{R},\ n(x)\in \mathbb{R},\ 4n(x)>t(x)^2\},$$
and the space of paravectors is defined by
$$\mathbb{R}^{(m+1)}:=\{x\in \mathbb{R}_m \ | \ [x]_k=0\ \text{for every}\  k>1\},$$
where $[x]_k$ denotes the k-vector pair of $x.$

We also set
$$\mathbb{S}_m:=\{J\in \mathcal{Q}_m\ |\ J^2=-1\}.$$
The elements of $\mathbb{S}_m$ are called the square roots of $-1$ in the Clilfford algebra $\mathbb{R}_m.$
It is known that
$$\mathbb{S}_m=\{x\in \mathbb{R}_m \ | \ t(x)=0, \ n(x)=1 \}.$$

We consider the cartesian product $\big(\mathbb{R}_m\big)^n$ of the Clifford algebra $\mathbb{R}_m.$
 Its complexification  is denoted by
$$\big(\mathbb{R}_m\big)^n_{\mathbb C}
=\big(\mathbb{R}_m\big)^n\otimes_{\mathbb{R}}{\mathbb{C}}:=\big(\mathbb{R}_m\big)^n+i\big(\mathbb{R}_m\big)^n.
$$

For each $x,y\in  \mathbb{R}^n,$
we define $$\overline{x+iy}=x-iy$$ be the complex conjugation of $x+iy$ in $\mathbb{R}^n_{\mathbb{C}}.$
Note that $\mathbb{R}_0=\mathbb{R}$ and $(\mathbb{R}_0)^n_{\mathbb{C}}=\mathbb{C}^n$,
then $\overline{x+iy}=x-iy,$ for each  $x,y\in \mathbb{R}^n.$

\begin{defn}\label{Defn-A}
	Let $D\subset\mathbb{C}^n$ be an open subset. A mapping $F:D\rightarrow \big(\mathbb{R}_m\big)^n_{\mathbb{C}}$ is called a $\mathbb{R}_m$-stem mapping over the domain $D$ if $F$ is complex intrinsic, i.e.  $$F(\overline{z})=\overline{F(z)}, \qquad \forall\ z\in D.$$
\end{defn}

\medskip

\begin{rem}\label{Rm-A}\
$(1)$\  A mapping $F$ is a $\mathbb{R}_m$-stem mapping if and only if the $\mathbb{R}_m$-valued components $F_1,F_2$ of the $F=F_1+iF_2$ form an even-odd pair, i.e.
$$F_1(\bar{z})=F_1(z), \qquad F_2(\bar{z})=-F_2(z)$$  for each $z\in D.$

$(2)$Consider $\mathbb{R}_m$ as a $2^m$-dimensional real vector space. By means of a basis $\mathcal{B}=\{e_A\}_{A\in \mathcal{P}_m}$ of $\mathbb{R}_m,$ $F$ can be identity with complex intrinsic curves in $\mathbb{C}^n.$

Let $F(z)=F_1(z)+iF_2(z)=\sum_{A\in \mathcal{P}_m} F^Ae_A$ with $F^A(z)\in \mathbb{C}^n.$ Then
$$\tilde F=\left(
\begin{array}{cccc}
F_1^0 & F_1^1  &\cdots & F_1^{2^m-1} \\
F_2^0 & F_2^1  &\cdots & F_2^{2^m-1} \\
\vdots& \vdots &\ddots & \vdots      \\
F_n^0 & F_n^1  &\cdots & F_n^{2^m-1} \\
\end{array}
\right):D\rightarrow \mathbb{C}^{n \times {2^m}}$$
satisfies $\tilde F(\bar z)= \overline{\tilde  F(z)}.$ Giving $\mathbb{R}_m$ the unique manifold structure as a real vector space, we get that a stem mapping $F$ is of class $C^k(k=0,\cdots,\infty)$ or real-analytic if and only if the same property for $\tilde F.$ This notion is clearly independent of the choice of the basis of $\mathbb{R}_m.$

$(3)$ In the case of $n=1$, Definition \ref{Defn-A} was first introduced   by Ghiloni and Petrotti \cite{Ghiloni2011002}.
\end{rem}

In the several Clifford numbers $(\mathbb R_m)^n$, we define
$$\big(\mathcal{Q}_m\big)^n_s  :=\bigcup_{\,I \in\mathbb{S}_m} \mathbb{C}_I^n,$$
where $$\mathbb C_I^n
:=\mathbb R^n +I \mathbb R^n, \qquad I\in \mathbb{S}_m.$$
Let  $\big(\mathbb{R}^{(m+1)}\big)^n$ be  the space of  multi-paravectors. Its subspace
$$\big(\mathbb{R}^{(m+1)}\big)^n_s:= \big(\mathbb{R}^{(m+1)}\big)^n \bigcap \big(\mathcal{Q}_m\big)^n_s$$
is called the space of  slice multi-paravectors.

For any $\alpha, \beta\in\mathbb R^n,\ I\in \mathbb S_m$ with $z=\alpha+\beta I$, we set
$$[z] :=\bigcup_{J\in\mathbb S_m} \alpha+\beta J    \subseteq \big(\mathcal{Q}_m\big)^n_s.$$

Given an open subset $D$ of $\mathbb{C}^n,$ let $\Omega_D$ be the subset of $\big(\mathcal{Q}_m \big)^n_s$ obtained by the action on $D$ of the square roots of $-1$:
$$\Omega_D:=\{\alpha+\beta J\in \big(\mathcal{Q}_m \big)^n_s \ |\  \alpha +i \beta \in D, \alpha ,\beta \in \mathbb{R}^n, \ J\in \mathbb{S}_m\},$$
and
$$D_I:=\Omega_D\cup \mathbb C_I^n,\qquad I\in \mathbb{S}_m,$$
then
$$\Omega_D=\bigcup_{I\in\mathbb S_m} D_I,\qquad I\in \mathbb{S}_m.$$
In particular, $\Omega_D=(\mathcal Q_m)^n_s$ when  $D=\mathbb C^n$.

\begin{defn}\label{DefA}
	Any stem mapping $F:D\rightarrow \big(\mathbb{R}_m\big)^n_{\mathbb{C}}$ induces a (left) slice mapping $$f=\mathcal{I}(F):\Omega_D \rightarrow \big(\mathbb{R}_m\big)^n.$$ If $x=\alpha+I\beta \in D_I,\ \forall I\in \mathbb{S}_m,$ we set $$f(x):=F_1 (z)+IF_2 (z) \qquad  (z=\alpha + i \beta ).$$
\end{defn}

The slice function $f$ is well defined, since $(F_1,F_2)$ is an even-odd pair w.r.t. $\beta$ and then
$f(\alpha+(-\beta)(-J))=F_1(\bar z)+(-J)F_2(\bar z)=F_1(z)+JF_2(z).$ There is an analogous definition for right slice mappings when the element $J\in \mathbb{S}_m$ is replace on the right of $F_2(z).$ In what follows, the term slice mapping will always mean left slice mappings.

We will denote the set of all (left) slice mappings on $\Omega_D$ by
\begin{equation*}
\begin{split}
\mathcal{S}\big(\Omega_D,\big(\mathbb{R}_m\big)^n\big)&:=\big\{f:\Omega_D\rightarrow \big(\mathbb{R}_m\big)^n\ |\ f=\mathcal{I}(F),\\& \qquad\qquad\qquad \qquad F:D\rightarrow \big(\mathbb{R}_m\big)^n_\mathbb{C}\ \text{is\ a}\ \mathbb{R}_m\text{-stem\ mapping}\big \}.
\end{split}
\end{equation*}

The distinguished property for a slice mapping is its  representation formula:

\begin{prop}\label{prop-1}
	Let $f\in \mathcal{S} \big(\Omega_ D,\big(\mathbb{R}_m\big)^n\big)$ and $J$,$K\in \mathbb{S}_m$ with $J\not=K$. Then
	$$f(\alpha+\beta I )=(I-K)((J-K)^{-1}f(\alpha+\beta J))-(I-J)((J-K)^{-1}f(\alpha+\beta K))$$
for each $ I  \in \mathbb{S}_m,\ \alpha, \beta \in \mathbb{R}^n$ with $\alpha+I \beta \in D_I.$
\end{prop}

\begin{proof}
	For $f\in \mathcal{S} \big(\Omega_ D,\big(\mathbb{R}_m\big)^n\big)$, by definition,
	$$ f(\alpha +\beta J)-f(\alpha-\beta K)=(J-K)F_2(\alpha+\beta i)$$
	Hence $$F_2(\alpha+\beta i)=(J-K)^{-1}(f(\alpha +\beta J)-f(\alpha+\beta K))$$
and \ 	\begin {equation*}
	\begin {split}
F_1(\alpha+\beta i)=&f(\alpha+\beta J)-JF_2(\alpha+\beta i)\\=&f(\alpha+\beta J)-J((J-K)^{-1}(f(\alpha +\beta J)-f(\alpha+\beta K))).
\end{split}
\end{equation*}
Therefore,
\begin {equation*}
\begin {split}
f(\alpha+\beta I )=&F_1(\alpha+\beta i)+IF_2(\alpha+\beta i)
\\=&f(\alpha+\beta J)+(I-J)((J-K)^{-1}(f(\alpha +\beta J)-f(\alpha+\beta K)))\\=&(J-K+I-J)((J-K)^{-1}f(\alpha +\beta J))+(I-J)((J-K)^{-1} f(\alpha+\beta K))\\=&(I-K)((J-K)^{-1}f(\alpha +\beta J))+(I-J)((J-K)^{-1} f(\alpha+\beta K)).
\end{split}
\end{equation*}
\end{proof}

By settting $K=-J$ in Proposition \ref{prop-1}, we obtain the following corollary.

\begin{cor}\label{cor-1}
	Let  $f\in \mathcal{S} \big(\Omega_ D,\big(\mathbb{R}_m\big)^n\big)$ and $ I \in \mathbb{S}_m,\ \alpha, \beta \in \mathbb{R}^n$ with $\alpha+I \beta \in \Omega_D.$   Then
	$$f(\alpha+\beta I )= \frac{1}{2}(f(\alpha+\beta J)+f(\alpha-\beta J))-\frac{I}{2}(J(f(\alpha+\beta J)-f(\alpha-\beta J))) $$
 for each $J\in\mathbb S_m.$
\end{cor}

We will denote by
$$\mathcal{S}^1\big(\Omega_D,\big(\mathbb{R}_m\big)^n\big):=\{f= \mathcal{I}(F) \in\mathcal{S }\big(\Omega_D,\big(\mathbb{R}_m\big)^n\big)|\ F\in C^{1}\big(D,\big(\mathbb{R}_m\big)^n_{\mathbb{C}}\big) \}$$
the real vector space of slice mappings of several Clifford variables with stem mapping of class $C^1.$

Let $f=\mathcal{I}(F)\in\mathcal{S}^1\big(\Omega_D,\big(\mathbb{R}_m\big)^n\big)$ and $z=\alpha+i\beta \in D.$ Then the partial derivatives $\partial F/\partial \alpha_t$ and $i\partial F/\partial \beta_t$ are continous $\mathbb{R}_m-$stem mappings on $D,$ for $t=1,2,\cdots,n.$ The same property holds for their linear combinations
$$\frac{\partial F}{\partial z_t}=\frac{1}{2}(\frac{\partial  F}{\partial \alpha_t}-i\frac{\partial F}{\partial \beta_t}) \qquad \text and \qquad \frac{\partial F}{\partial \overline z_t}=\frac{1}{2}(\frac{\partial F}{\partial \alpha_t}+i\frac{\partial F}{\partial \beta_t}),$$
where $z=(z_1,\cdots,z_n),\ \alpha=(\alpha_1,\cdots,\alpha_n),\ \beta=(\beta_1,\cdots,\beta_n),\ t=1,2,\cdots,n.$

\begin{defn}\label{rm-2}
	Let $f= \mathcal{I}(F) \in\mathcal{S }^1\big(\Omega_D,\big(\mathbb{R}_m\big)^n\big).$ We set
	$$\frac{\partial f}{\partial x}:= \mathcal{I}(\frac {\partial F}{\partial z}),\qquad \text and \qquad \frac{\partial f}{\partial \bar x}:= \mathcal{I}(\frac {\partial F}{\partial \bar z}),$$
i.e.
$$\big(\frac{\partial f}{\partial x_1},\cdots,\frac{\partial f}{\partial x_n}\big)=\big(\mathcal{I}(\frac {\partial F}{\partial z_1}),\cdots,\mathcal{I}(\frac {\partial F}{\partial z_n})\big),$$
and
$$\big(\frac{\partial f}{\partial \bar x_1},\cdots,\frac{\partial f}{\partial \bar x_n}\big)=\big(\mathcal{I}(\frac {\partial F}{\partial \bar z_1}),\cdots,\mathcal{I}(\frac {\partial F}{\partial \bar z_n} )\big),$$
where
$$\mathcal{I}(\frac{\partial F}{\partial z_t})=\big( \mathcal{I}(\frac {\partial F^1}{\partial z_t}),\ \mathcal{I}(\frac {\partial F^2}{\partial z_t}), \ \cdots,\ \mathcal{I}(\frac {\partial F^n}{\partial z_t})\big)^T,$$
and
$$\mathcal{I}(\frac {\partial F}{\partial \bar z_t})=\big( \mathcal{I}(\frac {\partial F^1}{\partial \bar z_t}),\ \mathcal{I}(\frac {\partial F^2}{\partial \bar z_t}),\ \cdots,\ \mathcal{I}(\frac {\partial F^n}{\partial \bar z_t})\big)^T,$$
for $F=(F^1,F^2,\cdots,F^n),\  t=1,2,\cdots,n, $ and $A^T$ stands for  the transpose of the vector $A.$
\end{defn}
The notation $\partial f/ \partial \bar x_t$ is justified by the following properties: $$\bar x =(\mathcal{I}(\bar z_1), \cdots,\mathcal{I}(\bar z_n))^T$$ and therefore $$\partial \bar x/\partial \bar x_t=(0,\cdots,0,1,0,\cdots,0)^T,\qquad \partial x/ \partial \bar x_t=0.$$

Left multiplication by $i$ defines a complex structure on $\big(\mathbb{R}_m\big)^n_{\mathbb{C}}.$ With respect to this structure, a $C^1$ mapping $F=F_1+iF_2:D\rightarrow \big(\mathbb{R}_m\big)^n_{\mathbb{C}}$ is holomorphic if and only if its components $F_1,$ $F_2$ satisfy the Cauchy-Riemann equations:
$$\frac{\partial  F_1}{\partial \alpha}=\frac{\partial  F_2}{\partial \beta},\qquad \frac{\partial  F_1}{\partial \beta}=-\frac{\partial  F_2}{\partial \alpha},$$
i.e.
$$\frac{\partial  F}{\partial \bar z}=0,\qquad (z=\alpha+i \beta \in D),$$
where
$$\frac{\partial F_p}{\partial \alpha}=\left(
\begin{array}{cccc}
\frac{\partial F_p^1}{\partial  \alpha_1} & \frac{\partial F_p^1}{\partial  \alpha_2} &\cdots & \frac{\partial F_p^1}{\partial  \alpha_n} \\
\frac{\partial F_p^2}{\partial  \alpha_1} & \frac{\partial F_p^2}{\partial  \alpha_2} &\cdots & \frac{\partial F_p^2}{\partial  \alpha_n}\\
\vdots& \vdots &\ddots & \vdots      \\
\frac{\partial F_p^n}{\partial  \alpha_1} & \frac{\partial F_p^n}{\partial  \alpha_2} &\cdots & \frac{\partial F_p^n}{\partial  \alpha_n} \\
\end{array}
\right),$$
and
$$\frac{\partial F_p}{\partial \beta}=\left(
\begin{array}{cccc}
\frac{\partial F_p^1}{\partial  \beta_1} & \frac{\partial F_p^1}{\partial \beta_2} &\cdots & \frac{\partial F_p^1}{\partial  \beta_n} \\
\frac{\partial F_p^2}{\partial  \beta_1} & \frac{\partial F_p^2}{\partial  \beta_2} &\cdots & \frac{\partial F_p^2}{\partial  \beta_n}\\
\vdots& \vdots &\ddots & \vdots      \\
\frac{\partial F_p^n}{\partial  \beta_1} & \frac{\partial F_p^n}{\partial  \beta_2} &\cdots & \frac{\partial F_p^n}{\partial  \beta_n} \\
\end{array}
\right),$$
for $F_p=(F_p^1,\cdots,F_p^n),\ p=1,2,$ and $\alpha=(\alpha_1,\cdots,\alpha_n),\ \beta=(\beta_1,\cdots,\beta_n).$

This condition is equivalent to require that, for any basis $\mathcal{B},$ the mapping $\tilde{F}$\ (cf. Remark \ref{Rm-A}) is holomorphic.

Set
$$\frac{\partial f}{\partial x}:=\left(
\begin{array}{cccc}
\frac{\partial f_1}{\partial x_1} & \frac{\partial f_1}{\partial x_2} &\cdots & \frac{\partial f_1}{\partial x_n} \\
\frac{\partial f_2}{\partial x_1} & \frac{\partial f_2}{\partial x_2} &\cdots & \frac{\partial f_2}{\partial x_n}\\
\vdots& \vdots &\ddots & \vdots      \\
\frac{\partial f_n}{\partial x_1} & \frac{\partial f_n}{\partial x_2} &\cdots & \frac{\partial f_n}{\partial x_n} \\
\end{array}
\right),
\qquad
\frac{\partial F}{\partial z}:=\left(
\begin{array}{cccc}
\frac{\partial F^1}{\partial z_1} & \frac{\partial F^1}{\partial z_2} &\cdots & \frac{\partial F^1}{\partial z_n} \\
\frac{\partial F^2}{\partial z_1} & \frac{\partial F^2}{\partial z_2} &\cdots & \frac{\partial F^2}{\partial z_n}\\
\vdots& \vdots &\ddots & \vdots      \\
\frac{\partial F^n}{\partial z_1} & \frac{\partial F^n}{\partial z_2} &\cdots & \frac{\partial F^n}{\partial z_n} \\
\end{array}
\right).$$
Similarly define $\frac{\partial f}{\partial \bar x}$ and $\frac{\partial F}{\partial \bar z}.$
Let
$$\mathcal{H}(D):=\big\{F\in C^1\big(D,\big(\mathbb{R}_m\big)^n_{\mathbb{C}}\big):\frac{\partial F}{\partial \bar z}=0, \ \forall z=(z_1,\cdots, z_n)\in D\big\}.$$

\begin{defn}
	A (left) slice mapping $f= \mathcal{I}(F) \in\mathcal{S }^1\big(\Omega_D,\big(\mathbb{R}_m\big)^n\big)$ is (left) slice regular if its stem mapping $F$ is holomorphic.
\end{defn}

We denote the vector space of slice regular functions on $\Omega_D$ by $\mathcal{SR}\big(\Omega_D,\big(\mathbb{R}_m\big)^n\big)$. That is,
		$$\mathcal{SR}\big(\Omega_D,\big(\mathbb{R}_m\big)^n\big):=\big\{\mathcal{I}(F)\in\mathcal{S}^1\big(\Omega_D,\big(\mathbb{R}_m\big)^n\big)\ |\  F:D\rightarrow \big(\mathbb{R}_m\big)^n_{\mathbb{C}}  \ \text {is\ holomorphic}\big\}.$$

\begin{rem} \label{rem-S2}
If we define
$$\mathbb{S}_m^*:=\{J\in \mathbb{R}^{(m+1)} \  | \ J^2=1\},$$
and
$$\Omega_D^*:=\{x=\alpha+\beta J\in \big(\mathbb{R}^{(m+1)} \big)^n_s \ |\  \alpha +i \beta \in D, \alpha ,\beta \in \mathbb{R}^n, \ J\in \mathbb{S}_m^*\},$$
then
$\Omega_D^*\subseteq \Omega_D$ and
$f\in \mathcal{SR}\big(\Omega_D^*,\big(\mathbb{R}_m\big)^n\big)$ is called the slice monogenic mapping.
\end{rem}

\begin{rem}
In the case of  $m=2$, we call $f\in \mathcal{SR}(\Omega_D,\mathbb{H}^n)$ a slice regular mapping of several quaternionic variables.
\end{rem}

\begin{prop}\label{prop:regular-derivative}
Let	$f= \mathcal{I}(F) \in\mathcal{S }^1\big(\Omega_D,\big(\mathbb{R}_m\big)^n\big).$ Then $f$ is slice regular on  $\Omega_D$ if and only if the restriction
$$f_I:D_I\rightarrow \big(\mathbb{R}_m\big)^n$$
is holomorphic for every $I\in \mathbb{S}_m.$
\end{prop}

\begin {proof}
Notice that $$f_I(\alpha+I\beta )=F_1(\alpha+i\beta)+IF_2(\alpha+i\beta).$$ If $F$ is holomorphic then
$$\frac{\partial f_I}{\partial \alpha}+I \frac{\partial f_I}{\partial \beta}=\frac{\partial F_1}{\partial \alpha}+I \frac{\partial F_2}{\partial \alpha}+I(\frac{\partial F_1}{\partial \beta}+I \frac{\partial F_2}{\partial \beta})=0$$
at every point $x=\alpha+I\beta \in D_I.$

Conversely, assume that $f_I$ is holomorphic at every $I\in \mathbb{S}_m.$ Then
$$0=\frac{\partial f_I}{\partial \alpha}+I \frac{\partial f_I}{\partial \beta}=\frac{\partial F_1}{\partial \alpha}-\frac{\partial F_2}{\partial \beta}+I( \frac{\partial F_2}{\partial \alpha}+\frac{\partial F_1}{\partial \beta} )$$
at every point $z=\alpha+i \beta \in D.$ From the arbitrariness of $I$ it follows that $F_1,F_2$ satisfy the Cauchy-Riemann equations.
\end{proof}

\begin{prop}
Let $f\in \mathcal{SR} \big(\Omega_ D,\big(\mathbb{R}_m\big)^n\big).$ For every choice of $I=I_1\in \mathbb{S}_m,$ let $I_2,\cdots,I_m$ be a completion to an orthonormal basis of the algebra $\mathbb{R}_m.$ Then there exists $2^{m-1}$ holomorphic mappings $F_A:D_I\rightarrow \mathbb{C}^n_I, A\in \mathcal{P}(m)$ such that for every $z=\alpha +I\beta \in D_I$ we have
$$f_I(z)=\sum_{A\in \mathcal{P}(m)} F_A(z)I_A.$$
Here $I_{\varnothing}=1.$
\end{prop}

\section{Growth theorems in the unit ball}

We consider the unit ball in the set of slice several Clifford variables $\big(\mathcal{Q}_m\big)^n_s,$ i.e.
$$\mathbb{B}:=\Big\{ x\in \big(\mathcal{Q}_m\big)^n_s  | \parallel x\parallel =\big(\sum_{t=1}^n|x_t|^2\big)^{\frac{1}{2}}<1, \ x=(x_1,\cdots,x_n)\Big\},$$
then
$$\mathbb{B}_I=\mathbb{B}\cap \mathbb{C}_I^n, \qquad  \forall\  I\in \mathbb{S}_m.$$

%In several complex variables function theory, for each $t=1,\cdots,n,$ let $g_t$ is a normalized starlike function on the unit disc $U$, if $\lambda_t \geq 0$ and $\sum_{t=1}^n\lambda_t=1,$ then
%$$S^*_\pi(\mathbb{B}_I):=\{f(z)=z\prod_{t=1}^n\big(\frac{g_t(z_t)}{z_t} \big)^{\lambda_t}|\ z=(z_1,\cdots,z_n)\in \mathbb{B}_I\}$$
%is a starlike mapping on $\mathbb{B}_I.$

\begin{lem}\label{lem-A}
Let $f \in \mathcal{S}\big(\Omega_D,\big(\mathbb{R}_m\big)^n\big),\ f({D_I})\subseteq \mathbb{C}^n_I$ for some $I\in \mathbb{S}_m.$ Then
$$\max\limits_{J\in \mathbb{S}_m}  ||f(\alpha+J\beta)||=\max\limits_{J=\pm I}||f(\alpha+J\beta)||,$$
$$\min\limits_{J\in \mathbb{S}_m}||f((\alpha+J\beta)||=\min\limits_{J=\pm I}||f(\alpha+J\beta)||,$$
for each $\alpha, \beta \in {\mathbb{R}^n},\ J\in \mathbb{S}_m$ with $\alpha+J\beta \in \Omega_D.$

\end{lem}

\begin{proof}
Since $f\in \mathcal{S}\big(\Omega_D,\big(\mathbb{R}_m\big)^n\big),$   the representation formula shows that
$$f(x)=\frac{1}{2}(f(z)+f(\bar z))-J\frac{I}{2}(f(z)-f(\bar z))$$
 for each $u,v\in\mathbb R^m$ and $J\in\mathbb S_m$ with $x=u+Jv\in D_J$ and $z=u+Iv$.

Denote
 \begin{eqnarray*}\alpha&=&\frac{1}{2}(f(z)+f(\bar z)), \\
  \beta&=&-\frac{I}{2}(f(z)-f(\bar z)).
  \end{eqnarray*}
By assumption, we have both $\alpha$ and $\beta$ in $\mathbb{C}^n_I$.
We set
 $$\alpha=(\alpha_1,\cdots,\alpha_n), \qquad \beta=(\beta_1,\cdots,\beta_n).$$
For  $\alpha_t\not =0,$ we set $$\beta_t \alpha_t^{-1}=a_t+Ib_t,$$
where $a_t, b_t\in \mathbb{R}$ for any $t\in\{1,\cdots,n\}.$

Take $I_2,\cdots,I_{2^m-1}\in \mathbb{S}_m$ such that $\{1,I,I_2,\cdots,I_{2^m-1}\}$ consists of a basis of $\mathbb{R}_m.$ For any $J\in\mathbb S_m$
we can represent  it under such a basis  as  $$J=uI+\sum_{l=2}^{2^m-1}v_lI_l$$ with  real coefficients  $u,v_2,v_3,\cdots,v_{2^m-1} \in \mathbb{R}$  such that  $$ u^2+\sum_{l=2}^{2^m-1}v_l^2=1.$$
We can rewrite  $$f(x)=\alpha+J\beta,$$
then
\begin {equation*}
\begin {split}
||f(x)||^2=&\sum_{t=1}^n |\alpha_t+J\beta_t|^2
\\=&\sum_{t=1,\alpha_t\not =0}^n|1+J \beta_t \alpha_t^{-1}|^2|\alpha_t|^2+\sum_{t=1,\alpha_t=0}^n|\beta_t|^2
\\=&\sum_{t=1,\alpha_t\not =0}^n|1+(uI+\sum_{l=2}^{2^m-1}v_lI_l) (a_t+Ib_t)|^2|\alpha_t|^2+\sum_{t=1,\alpha_t=0}^n|\beta_t|^2
\\=&\sum_{t=1,\alpha_t\not =0}^n(1+a^2_t+b_t^2-2b_tu)|\alpha_t|^2+\sum_{t=1,\alpha_t=0}^n|\beta_t|^2
\\=&\sum_{t=1,\alpha_t\not =0}^n(1+a^2_t+b_t^2)|\alpha_t|^2+\sum_{t=1,\alpha_t=0}^n|\beta_t|^2-\sum_{t=1}^n 2b_t|\alpha_t|^2u
\\=&:g(u)
\end{split}
\end{equation*}

Therefore $$||f(x)||_{\max}=\max_{u\in[-1,1]}g(u)=\max_{u=\pm 1}g(u),$$
 $$||f(x)||_{\min}=\min_{u\in[-1,1]}g(u)=\min_{u=\pm 1}g(u).$$

\end{proof}

\begin{rem}
In \cite{Wang2017001}, it obtained the similar result to Lemma \ref{lem-A} in the case of $n=1.$ The method above is simpler and satisfies in higher dimensions. Moreover, f does
not need to be a slice monogenic mapping on $\Omega_D^*$(see Remark \ref{rem-S2}), it only need to be a slice mapping on $\Omega_D$.
%it concluded that��
%Let $f$ be a slice monogenic function on a symmetric slice domain $U\subseteq \mathbb{R}^{(n+1)}$ such that
%$f(U_I)\subseteq \mathbb{C}_I$ for some $I\in \mathbb{S}.$ Then for each sphere $u+v \mathbb{S}\subseteq U,$
%we  have equalities:
%$$\max_{J\in \mathbb{S}}|f(u+vJ)| =\max{|f(u+vI)|,|f(u-vI)|} ,$$
%$$\min_{J\in \mathbb{S}}|f(u+vJ)| =\min{|f(u+vI)|,|f(u-vI)|} .$$
\end{rem}

\begin{thm}\label{thm-c}
Let $f$ be a mapping in $\mathcal{SR}\big(\mathbb{B},\big(\mathbb{R}_m\big)^n\big)$ such that its restriction $f_I$ to $\mathbb{B}_I$ is a starlike mapping such that $f({\mathbb{B}_I})\subseteq \mathbb{C}_I^n$ for some $I\in \mathbb{S}_m.$ If $f(0)=0, f'(0)=\mathbb{I}_n$ (Identity matrix of order n), then
$$\frac{\parallel x\parallel}{(1+\parallel x\parallel )^2}\leq \parallel f(x)\parallel \leq \frac{\parallel x\parallel}{(1-\parallel x\parallel )^2}, \qquad \forall\ x\in \mathbb{B}.$$
These estimates are sharp.
\end{thm}

\begin{proof}
Since $f_I:\mathbb{B}_I \subset \mathbb{C}_I^n \rightarrow \mathbb{C}_I^n$ is a starlike mapping such that $f_I(0)=0, f'_I(0)=\mathbb{I}_n,$
we have \cite{Graham2003001}
$$\frac{\parallel z\parallel}{(1+\parallel z\parallel )^2}\leq \parallel f_I(z)\parallel \leq \frac{\parallel z\parallel}{(1-\parallel z\parallel )^2}\qquad \forall\ z\in \mathbb{B}_I.$$
With $z$ replaced by $\bar z$, then
$$\frac{\parallel z\parallel}{(1+\parallel z\parallel )^2}\leq \parallel f_I(\overline z)\parallel \leq \frac{\parallel z\parallel}{(1-\parallel z\parallel )^2}\qquad \forall\ \overline z\in \mathbb{B}_I.$$
These estimates are sharp.

For any $x\in \mathbb B$, there exists $z=\alpha+I\beta \in \mathbb{B}_I$ and $ J\in \mathbb{S}_m$  such that  $x=\alpha+J\beta \in \mathbb{B}$��
Lemma \ref{lem-A} concludes  that
$$||f(x)||\geq \min\{||f_I(z)||,||f_I(\overline{z})||\}\geq \frac{\parallel z\parallel}{(1+\parallel z\parallel )^2}
=\frac{\parallel x\parallel}{(1+\parallel x\parallel )^2}.$$
This estimate is sharp.

The reverse  inequality can be proved similarly.
\end{proof}

%\begin{thm}
%Let $f\in \mathcal{SR}\big(\mathbb{B},\big(\mathbb{R}_m\big)^n\big)$ such that its restriction $f({\mathbb{B}_I})\subset \mathbb{C}_I^n$ such that $f_I\in S^*_\pi(\mathbb{B}_I)$ for some $I\in \mathbb{S}_m,$ then
%$$\frac{1-\parallel x\parallel}{(1+\parallel x\parallel )^{2n+1}}\leq |J_f(x)| \leq \frac{1+\parallel x\parallel}{(1-\parallel x\parallel )^{2n+1}}\qquad x\in \mathbb{B}_I.$$
%These estimates are sharp.
%\end{thm}
%\begin{proof}
%The result follows directly from the \cite{Graham2003001}.
%\end{proof}

\begin{exa} We consider the function
Let $f:\mathbb{B}\rightarrow (\mathbb{R}_m\big)^n$, defined by
$$f(x)=\big(x_1(1-x_1e^{I\theta})^{-*2},\cdots,x_n(1-x_ne^{I\theta})^{-*2}\big).$$
%
%We denote
%$$[x]=\{\alpha+K\beta \ |\ x=\alpha+I\beta, \alpha,\beta \in \mathbb{R}^n,\ I,K \in \mathbb{S}_m\}, \qquad \forall\ x\in\mathbb B.$$
%For any $J\in \mathbb{S}_m$,  we consider  $$z=(z_1,z_2,\cdots,z_n)\in [x]\cap \mathbb{C}_J^n.$$
%Notice that
%$$f_J(z)=\big(z_1(1-z_1e^{I\theta})^{-2},\cdots,z_n(1-z_ne^{I\theta})^{-2}\big).$$
%By direct calculation, we have
%$$\frac{\partial f_J(z)}{\partial \bar z}=0.$$
%so that
It is easy to see that $f\in \mathcal{SR}\big(\mathbb{B},\big(\mathbb{R}_m\big)^n\big),\ f(\mathbb{B}_I)\subseteq \mathbb{C}_I^n$   and $f(0)=0,\ f'(0)=I_n.$
Since $f_I$ is a starlike mapping on $\mathbb{B}_I$ (see \cite{Graham2003001}),
we have  $f$ satisfies all the conditions of Theorem \ref{thm-c}.
\end{exa}

For convex mappings, the same approach yields the similar results.

\begin{thm}
Let $f$ be a mapping in $\mathcal{SR}\big(\mathbb{B},\big(\mathbb{R}_m\big)^n\big)$ such that its restriction $f_I$ to $\mathbb{B}_I$ is a convex mapping such that $f({\mathbb{B}_I})\subseteq \mathbb{C}_I^n$ for some $I\in \mathbb{S}_m.$ If $f(0)=0, f'(0)=\mathbb{I}_n$, then
$$\frac{\parallel x\parallel}{1+\parallel x\parallel }\leq \parallel f(x)\parallel \leq \frac{\parallel x\parallel}{1-\parallel x\parallel }, \qquad\forall\  x\in \mathbb{B}.$$
These estimates are sharp.
\end{thm}

\begin{thm}\label{thm-a}
Let $f\in \mathcal{SR}\big(\mathbb{B},\big(\mathbb{R}_m\big)^n\big)$ such that its restriction $f_I$ to $\mathbb{B}_I$ is a convex mapping such that $f({\mathbb{B}_I})\subseteq \mathbb{C}_I^n$ for some $I\in \mathbb{S}_m.$ If $f(0)=0, f'(0)=\mathbb{I}_n$ ( Identity matrix of order n), then
$$\frac{\parallel x\parallel}{1+\parallel x\parallel }\leq \parallel f(x)\parallel \leq \frac{\parallel x\parallel}{1-\parallel x\parallel }\qquad x\in \mathbb{B}.$$
These estimates are sharp.
\end{thm}

\begin{exa}
The function  $f:\mathbb{B}\rightarrow (\mathbb{R}_m\big)^n$, defined by
$$f(x)=\big(x_1(1-x_1e^{I\theta}),\cdots,x_n(1-x_ne^{I\theta})\big), $$
provides an example satisfying all the conditions of Theorem \ref{thm-a}.
\end{exa}

\section{Growth theorems in the starlike or convex domains}

We shall now generalize the results in the last section to more general domains, other than the unit ball.

We consider the  bounded slice domains which are slice starlike and slice circular.

\begin{defn}
A set $\Omega \in \big(\mathcal{Q}_m\big)^n_s$ is called starlike with respect to a fixed point
$\omega_0 \in \Omega$ if the closed line segment joining $\omega_0$ to each
point $\omega\in \Omega$ lies entirely in $\Omega$. Also we say that $\Omega$ is convex if for all $\omega_1,\ \omega_2 \in \Omega,$ the closed line segment between $\omega_1$ and $\omega_2$ lies entirely in $\Omega$. In other words, $\Omega$ is convex if and only if $\Omega$ is starlike with respect to each of its points. The term starlike will mean starlike with respect to zero.
\end{defn}

\begin{defn}
A set $\Omega \in \big(\mathcal{Q}_m\big)^n_s$ is called slice starlike if $\Omega_I$ is starlike in $\mathbb{C}_I^n$ for some $I\in \mathbb{S}_m.$
\end{defn}

\begin{defn}
A domain $\Omega \in \big(\mathcal{Q}_m\big)^n_s$ is called slice circular, if for any $z\in \Omega_I,\ \theta \in \mathbb{R},$ and $I, J\in \mathbb{S}_m,$ we have $e^{J\theta}x \in \Omega$ for any  $x\in [z]\cap {\mathbb{C}^n_J}$.\end{defn}

\begin{rem}
A domain $\Omega$ is slice circular if and only if for some $I\in \mathbb{S}_m,$

$(1)\ \Omega_I$ is circular $(i.e.\ e^{I\theta}z\in \Omega_I$ if  $z\in \Omega_I$ and $\theta \in \mathbb{R})$;

$(2)\ \Omega$ is axially symmetric $(i.e.\ x\in \Omega$ if  $x\in [z]$ and $z\in \Omega_I)$.
\end{rem}

 Obviously, the unit ball $\mathbb{B}$ and polydisc $\mathbb{P}$ are slice circular, where
 $$\mathbb{P}:=\big\{ x\in \big(\mathcal{Q}_m\big)^n_s \big| \  |x_t|<1, \ x=(x_1,\cdots,x_n), \ t=1,2,\cdots,n\big\}.$$

\begin{defn}
A domain $\Omega \in \big(\mathcal{Q}_m\big)^n_s$ is called slice domain, if

$(1)\ \Omega \cap \mathbb{R}^n \not = \emptyset,$

$(2)\ \Omega_I$ is a domain of $\big(\mathcal{Q}_m\big)^n_s \cap \mathbb{C}_I^n,$ for any $I \in \mathbb{S}_m.$
\end{defn}

The bounded slice starlike slice circular and slice domain has an analytic characterization via definition functions.

\begin{lem}\label{B1} An axially symmetric slice domain
$\Omega \subseteq \big(\mathcal{Q}_m\big)^n_s$ is bounded slice starlike and slice circular  if and only if there exists a unique continuous function $$\rho: \big(\mathcal{Q}_m\big)^n_s \rightarrow \mathbb{R},$$ called the defining function of $\Omega,$ such that

$(1)\ \rho(x)\geq 0,\ \forall x\in \big(\mathcal{Q}_m\big)^n_s;\ \rho(x)=0 \Longleftrightarrow x=0$;

$(2)\ \rho(tx)=|t|\rho(x),\  \forall J\in \mathbb{S}_m, \forall x\in \mathbb{C}_J^n,\ t\in \mathbb{C}_J,\  $;

$(3)\ \Omega=\{x\in \big(\mathcal{Q}_m\big)^n_s:\ \rho(x)<1\}.$
\end{lem}

\begin{proof}
The proof is similar as the case of several complex variables \cite{Liu1998001}.
If the continuous function $\rho(x)$ satisfies $(1),\ (2),$ and $(3),$ then clearly $\Omega$ is a starlike slice circular domain. Its boundedness is also easy to prove, since if there exists a ray coming from the origin which completely falls in the starlike domain $\Omega$, then for any fixed point $x_0$ in this ray we have from $(3)$
$$\ \rho(tx_0)<1,\quad \forall\  t\in [0, \infty).$$  But we then obtain $\rho(x_0)=0$ in terms of  $(2)$. This contradicts $(1).$

Conversely, if $\Omega$ is a bounded slice starlike and slice circular domain in $\big(\mathbb{R}^{m+1}\big)_s^n,$ then we define
$$\rho(x)=\inf\{c>0: \ \ c^{-1}x\in \Omega\}.$$
Obviously, $\rho(x)$ satisfies  $(2),\ (3),$ and $\rho(x)\geq 0$ for all $x$.  If there exists a point $x_0\not =0,$ such that $\rho(x)=0,$ then it follows from  $(2)$ and $(3)$  that $\Omega$ includes the whole ray which comes from the origin and through the point $x_0,$ hence $\Omega$ is unbounded. Thus $(1)$ holds. Finally we prove that $\rho(x)$ is continuous. Clearly,
$$\{x\in \big(\mathcal{Q}_m\big)^n_s:\ r<\rho(x)<R\ \}=R\Omega\setminus \overline{r\Omega}$$
is an open set in $\big(\mathcal{Q}_m\big)_s^n,$ which implies the continuity of $\rho.$
\end{proof}

\begin{rem}\label{rem-S}
From the proof of  Lemma \ref{B1}, we have that
$$\rho(x)=\rho(z),\qquad \forall z\in [x]\subseteq \Omega,$$
where $\Omega$ is a bounded slice starlike and slice circular and slice domain.
\end{rem}

Now we can state our main result about the growth theorem in bounded starlike slice circular and slice domain.

\begin{thm}
Let $\Omega_D$ be a bounded slice starlike slice circular slice domain in $\big(\mathcal{Q}_m\big)^n_s,$ its defining function $\rho(x)$ is a $C^1$ function on $\Omega_D$ except for a lower dimensional set. If $f\in \mathcal{SR}\big(\Omega_D,\big(\mathbb{R}_m\big)^n\big)$ such that its restriction $f_I$ to $D_I$ is a starlike mapping such that $f(D_I)\subseteq \mathbb{C}_I^n$. If $f(0)=0, f'(0)=\mathbb{I}_n$ ( Identity matrix of order n), then for any $x\in \Omega_D,$
$$\frac{\rho(x)}{(1+\rho(x))^2 }\leq \parallel f(x)\parallel \leq \frac{\rho(x)}{(1-\rho(x))^2}$$
or equivalently,
$$\frac{\parallel x\parallel}{(1+\rho(x))^2 }\leq \parallel f(x)\parallel \leq \frac{\parallel x\parallel}{(1-\rho(x))^2}.$$
These estimates are sharp.
\end{thm}

\begin{proof}
Since $\Omega_D$ be a bounded slice starlike slice circular domain in $\big(\mathcal{Q}_m\big)^n_s,$ $D_I$ is bounded, complex circular in $\mathbb{C}_I^n,$ and is also starlike in $\mathbb{C}_I^n,$ for any $I\in \mathbb{S}_m.$ By the assumption that $\rho(x)$ is a $C^1$ function on $\Omega_D$, we know that $\rho(z)$ is a $C^1$ function on $\mathbb{C}_I^n$. Notice that  $f_I$ is a biholomorphic starlike mapping on $D_I$  and admits the properties that  $f(D_I)\subseteq \mathbb{C}_I^n$ and $f(0)=0, f'(0)=\mathbb{I}_n$. Thanks to the classical results in several complex variables
\cite{Liu1998001},  we have
$$\frac{\rho(z)}{(1+\rho(z))^2 }\leq \parallel f(z)\parallel \leq \frac{\rho(z)}{(1-\rho(z))^2}, \qquad \forall\ z\in D_I,$$
as well as
$$\frac{\parallel z\parallel}{(1+\rho(z))^2 }\leq \parallel f(z)\parallel \leq \frac{\parallel z\parallel}{(1-\rho(z))^2}\qquad \forall\ z\in D_I.$$
By symmetrically, we also have
$$\frac{\rho(\bar z)}{(1+\rho(\bar z))^2 }\leq \parallel f(\bar z)\parallel \leq \frac{\rho(\bar z)}{(1-\rho(\bar z))^2}, \qquad \forall\ z\in D_I.$$
These estimates are sharp.

Suppose $x=\alpha+J\beta \in \Omega_D$ for any $J\in \mathbb{S}_m$ and $z=\alpha+I\beta \in D_I$.
By Lemma \ref{lem-A}, we have
$$||f(x)||\geq \min\{||f_I(z)||,||f_I(\overline{z})||\}\geq \frac{\rho(z)}{(1+\rho(z))^2 }
=\frac{\rho(x)}{(1+\rho(x))^2 }.$$
These estimates are sharp.
The remaining inequality can be proved similarly.
\end{proof}

Likewise, we can establish the growth theorem in bounded convex slice circular domain for convex mappings.

Notice that since the convex domain is a starlike domain, for a bounded convex slice circular and slice domain $\Omega_D,$ there also exists a definition function of the domain $\Omega_D.$

\begin{thm}
Let $\Omega_D$ be a bounded convex slice circular slice domain in $\big(\mathcal{Q}_m\big)^n_s,$ with defining function $\rho(x).$ If $f\in \mathcal{SR}\big(\Omega_D,\big(\mathbb{R}_m\big)^n\big)$ such that its restriction $f_I$ to $D_I$ is a convex mapping such that $f(D_I)\subseteq \mathbb{C}_I^n$. If $f(0)=0, f'(0)=\mathbb{I}_n$  (identity matrix of order n), then for any $x\in \Omega_D,$
$$\frac{\rho(x)}{1+\rho(x)}\leq \parallel f(x)\parallel \leq \frac{\rho(x)}{1-\rho(x)}$$
or equivalently,
$$\frac{\parallel x\parallel}{1+\rho(x)}\leq \parallel f(x)\parallel \leq \frac{\parallel x\parallel}{1-\rho(x)}.$$
These estimates are sharp.
\end{thm}

\section{Final remarks}

In this paper, we have proved the growth theorem for slice monogenic extensions of starlike and convex mappings on the unit ball in the subset of slice several Clifford variables $\big(\mathcal{Q}_m\big)^n_s.$ However,  the corresponding  distortion theorem is still untouched. This deserves further investigation.

\newpage

\nocite{*}
\bibliographystyle{plain}
\bibliography{mybibfile.bib}
\end{document}